\documentclass[11pt]{article}
\usepackage[utf8]{inputenc}
\usepackage{amsfonts}
\usepackage{amsmath} 
\usepackage{amssymb}
\usepackage{amsthm}
\usepackage{tikz}
\usepackage[margin=.85in]{geometry}
\usepackage{mathtools}

\newtheorem{theorem}{Theorem}[section]
\newtheorem{lemma}[theorem]{Lemma}
\newtheorem{proposition}[theorem]{Proposition}
\newtheorem{corollary}[theorem]{Corollary}
\newtheorem*{theerem}{Theorem}
\newtheorem{definition}[theorem]{Definition}

\newtheorem{remark}[theorem]{Remark}

\newcommand{\bth}{\begin{theorem}}
\renewcommand{\eth}{\end{theorem}}
\newcommand{\bpr}{\begin{proposition}}
\newcommand{\epr}{\end{proposition}}
\newcommand{\bde}{\begin{definition}}
\newcommand{\ede}{\end{definition}}
\newcommand{\blem}{\begin{lemma}}
\newcommand{\elem}{\end{lemma}}
\newcommand{\bco}{\begin{corollary}}
\newcommand{\eco}{\end{corollary}}
\newcommand{\prove}{\begin{proof}}
\newcommand{\done}{\end{proof}}

\newcommand{\ite}{\begin{itemize}}
\newcommand{\mize}{\end{itemize}}

\newcommand{\ben}{\begin{enumerate}}
\newcommand{\een}{\end{enumerate}}

\DeclareMathOperator{\st}{st}

\DeclareMathOperator{\Dim}{Dim}

\DeclareMathOperator{\val}{val}
\DeclareMathOperator{\first}{first}
\DeclareMathOperator{\last}{last}
\DeclareMathOperator{\LRM}{LRM}
\DeclareMathOperator{\des}{des}
\DeclareMathOperator{\ap}{ap}
\DeclareMathOperator{\cp}{cp}
\DeclareMathOperator{\fix}{fix}
\DeclareMathOperator{\inv}{inv}
\DeclareMathOperator{\maj}{maj}
\DeclareMathOperator{\pos}{pos}
\DeclareMathOperator{\block}{bl}
\DeclareMathOperator{\spread}{sp}
\DeclareMathOperator{\SB}{SB}
\DeclareMathOperator{\Av}{Av}

\title{Set Partition Patterns and the Dimension Index}
\author{Thomas Grubb\\\texttt{tgrubb@ucsd.edu}\\University of California, San Diego \and Frederick Rajasekaran\\\texttt{frajasek@ucsd.edu}\\University of California, San Diego}
\date{\today\\[10pt]
	\begin{flushleft}
	\small Key Words: Avoidance, dimension index, generating functions, inversions, patterns, permutations, set partitions 
	                                       \\[5pt]
	\small AMS subject classification (2010):  05A18, 05A19
	\end{flushleft}}
\begin{document}
\maketitle
\abstract{The notion of containment and avoidance provides a natural partial ordering on set partitions. Work of Sagan and of Goyt has led to enumerative results on avoidance classes of set partitions, which were refined by Dahlberg et al. through the use of combinatorial statistics. We continue this work by computing the distribution of the dimension index (a statistic arising from the supercharacter theory of finite groups) across certain avoidance classes of partitions. In doing so we obtain a novel connection between noncrossing partitions and $321$-avoiding permutations, as well as connections to many other combinatorial objects such as Motzkin and Fibonacci polynomials.}
\section{Introduction}
Given a finite set $S$, a \emph{set partition} of $S$ is a unordered collection of disjoint nonempty blocks $B_1,\dots,B_k$ such that 
$$
\bigcup_{j=1}^kB_j=S.
$$
We will write a set partition $\pi$ as 
$$
\pi=B_1/\dots/B_k,
$$
and use the notation $\pi\vdash 
S$ to mean $\pi$ is a partition of $S$. We will restrict our attention in this article to set partitions of $[n]=\{1,2,\dots,n\}$ for some positive integer $n$; in this case we will omit set braces in the blocks of a partition, and write the blocks in \emph{standard order}, meaning 
$$
\min (B_1)<\dots<\min(B_k).
$$
We let $\Pi_n$ denote the collection of set partitions of $[n]$; thus $\Pi_3$ consists of the partitions 
$$
123,1/23,12/3,13/2,\text{ and }1/2/3.
$$

Set partitions form a foundational topic in combinatorics; see the books of Mansour \cite{Man13} or Stanley \cite{Sta12} for a general reference. In addition to having a rich combinatorial history, the theory of set partitions arises in the study of stochastic processes \cite{Pit06}, algebras \cite{HR05}, Hopf algebras \cite{AM10}, and many other areas.

The combinatorics of set partitions have recently been enhanced through the use of \emph{patterns}. Given a set partition $\pi = B_1/\dots/B_k$ of $[n]$ and a subset $S\subset [n]$, let $\pi\cap S$ be the partition of $S$ given by taking the nonempty intersections of the form $B_i\cap S$, $i=1,2,\dots, k$. We \emph{standardize} $\pi\cap S$ to obtain a partition $\st(\pi\cap S)$ of $[|S|]$ by replacing the $i$th smallest entry in $\pi\cap S$ with $i$. For example, if $n=6$, $\pi = 12/345/6$, and $S= \{2,3,5\}$, then $\pi\cap S = 2/35$ and $\st(\pi\cap S) = 1/23$.

Given two set partitions $\sigma\vdash[k]$ and $\pi\vdash[n]$, we say that $\pi$ \emph{contains} $\sigma$ as a pattern if there exists a subset $S\subset [n]$ with $\st(\pi\cap S) = \sigma$. If no such subset exists, then $\pi$ \emph{avoids} $\sigma$. Continuing our previous example, $\pi = 12/345/6$ contains the pattern $\sigma = 1/23$, but $\pi$ avoids the pattern $\tau = 1/2/3/4$ because $\pi$ only contains $3$ blocks.

The theory of set partition patterns can be traced in part back to work of Kreweras \cite{Kre72}, and was developed more generally in work of Klazar \cite{Kla96},\cite{Kla00A},\cite{Kla00B}. A fundamental question in the area, in analog to the related question in permutation patterns, is to enumerate the number of partitions of $[n]$ which avoid a set of fixed patterns. Namely, given a collection of partitions $P$, let 
$$
\Pi_n(P) = \{\pi \in \Pi_n: \pi \text{ avoids every partition in }P\}.
$$
Work of Klazar and Marcus \cite{KM07} has given asymptotic formulae for the sizes of these sets, and exact enumerative results for small patterns have been provided by Sagan \cite{Sag10} and Goyt \cite{Goy08}.

Again in analogy with permutation patterns, work has been devoted to refining these enumerative results through the use of combinatorial statistics. Often the statistic of interest is related to ``four fundamental statistics'' of Wachs and White \cite{WW91}. This can be found in work of Simion \cite{Sim94}, Goyt and Sagan \cite{GS09}, Dahlberg et al. \cite{REUPartition}, Lin and Fu \cite{LF17}, and Acharyya, Czajkowski, and Williams \cite{ACW20}.

The purpose of this paper is to continue the above work using a statistic arising from the supercharacter theory of finite groups. Building on work of Andr{\'e} \cite{And95}, \cite{And99}, \cite{And02} and Yan \cite{Yan01}, Diaconis and Isaacs have commenced a study of the representation theory of finite algebra groups through the use of \emph{supercharacters} and \emph{superclasses} \cite{DI08}. We will not give an introduction to this theory here, but instead will simply say that often such supercharacters can be indexed by set partitions with additional data; in particular, the \emph{dimension index} of such a partition $\pi=B_1/\dots/B_k$, given as a sum 
$$
\dim(\pi) = \sum_{i=1}^k(\max(B_i)-\min(B_i)+1),
$$
returns algebraic data regarding the corresponding supercharacter. 

As a result of this theory, there has been recent interest in the combinatorics of the dimension index on set partitions. For instance, Chern, Diaconis, Kane, and Rhoades have shown that this statistic satisfies an asymptotic central limit theorem \cite{CDKR15}, and that the average of this statistic (taken over $\Pi_n$) can be expressed quite cleanly in terms of the Bell numbers \cite{CDKR14}.

The purpose of this work is to combine the study of the dimension index and set partition patterns. In doing so, we give numerous refinements of enumerative results, and also find connections to other combinatorial objects. As a sample of one such connection, we are able to give a quick alternative proof of Theorem 8.4 of \cite{CEKS13} regarding $321$-avoiding permutations (definitions for those unfamiliar may be found in Section 5):
\begin{theerem}
Let $I_n(q,t,x)$ be the generating function for left-to-right maxima, inversions, and fixed points, taken over the $321$-avoiding permutations of length $n$:
$$
I_n(q,t,x) = \sum_{\pi \in Av_n(321)}q^{\inv(\pi)}t^{\LRM(\pi)}x^{\fix(\pi)}.
$$
Then $I_0(q,t,x) = 1$ and, for $n\geq 1$, 
$$
I_n(q,t,x) = tx I_{n-1}(q,t,x) + \sum_{j=2}^nq^{j-1}I_{j-2}(q,t,1)\big(I_{n-j+1}(q,t,x) - t(x-1)I_{n-j}(q,t,x)\big).
$$
\end{theerem}
Of note is that the proof of this theorem in \cite{CEKS13} requires algebraic manipulation of continued fractions, whereas our proof of this theorem is purely combinatorial. It relies only on a Catalan recursion and the Inclusion-Exclusion Principle.

To aid our study, we introduce the following generating functions: for a set of partitions $P$ and a variable $q$, we define
$$
\Dim_n(P;q) = \sum_{\pi\in \Pi_n(P)}q^{\dim(\pi)}.
$$
In fact, it will be simpler (and more interesting) to use two related statistics. Namely, for $\pi = B_1/\dots/B_k$ a partition of $[n]$, define the \emph{spread} and \emph{block} statistics by 
\begin{align*}
\spread(\pi) &= \sum_{i=1}^k (\max(B_i)-\min(B_i))\\
\block(\pi) &= k,
\end{align*}
so that $\dim(\pi) = \spread(\pi)+\block(\pi)$. In particular, we will examine the joint distributions of spread and block by calculating the joint generating function
$$
\SB_n(P;q,t) = \sum_{\pi\in \Pi_n(P)}q^{\spread(\pi)}t^{\block(\pi)},
$$
from which we can obtain the desired information on $\Dim_n(P;q)$ by equating the variables $q$ and $t$.

The outline of this paper is as follows. In Section 2, we start by recalling basic facts on set partitions which will be useful in our proofs, such as the connection between partitions and restricted growth functions. We then use this to study partitions which avoid a single pattern of length 3. Section 3 is devoted to partitions avoiding the pattern $13/24$; these are the so called \emph{noncrossing partitions}. Section 4 is devoted to avoidance of multiple patterns. We end by connecting noncrossing partitions to 321-avoiding permutations in Section 5, and with ideas for future work in Section 6.

\textbf{Acknowledgements:} Much of this work was aided by computations done in SageMath \cite{Sage} combined with use of the Online Encyclopedia of Integer Sequences \cite{OEIS}; we thank the contributors to both projects. Both authors would like to gratefully acknowledge financial support from NSF grant DMS-1849173. Additionally, the first listed author would like to acknowledge financial support from NSF Research Training Group grant DMS-1502651.
\section{Background and Introductory Calculations}
\subsection{Preliminary Notions}
We start this section by recalling several preliminary facts which will be useful in the study of set partition patterns. The first notion is that of a \emph{restricted growth function}, or \emph{RGF} for short; it allows us to frame containment questions in terms of words and subwords, which simplifies many arguments. 

A restricted growth function is a sequence $w = a_{1}a_{2}\dots a_{n}$ of positive integers such that $a_1 = 1$ and, for $i \geq 2$,
\begin{equation*}
a_i \leq 1 + \max\{a_1, \dots, a_{i-1}\}.
\end{equation*}

Let $R_n$ denote the set of length $n$ RGFs. It is straightforward to show that the following gives a bijection between $\Pi_n$ and $R_n$. Given $\pi = B_1/\dots/B_k\in \Pi_n$ written in standard order, we map $\pi$ to the RGF $w(\pi) = w_1\dots w_n$ with
$$
w_i = j
$$
if $i\in B_j$.
For example, the partition $\pi = 14/25/378/6 \vdash [8]$ maps to the RGF $12312433$ under this bijection.

Mapping partitions to their associated RGFs often provides a useful characterization of avoidance classes of set partitions. For a set of partitions $P$, let 
$$
R_{n}(P) = \{w(\pi) \in R_{n} : \pi \in \Pi_n(P)\}.
$$
Below we collect several results from Sagan's article \cite{Sag10}, which provide characterizations of $R_n(P)$ where $P$ consists of a single pattern of length $3$.

\bth[\cite{Sag10}]
We have the following characterizations.
\begin{enumerate}
    \item $R_n(1/2/3) = \{w \in R_n : \text{w consists of only 1s and 2s}\}.$
    \item $R_n(1/23) = \left\{w \in R_n : \begin{array}{c}\text{w is obtained by inserting a single 1 into a word of} \\ \text{the form $1^l23\dots m$ for some $l \geq 0$ and $m \geq 1$}\end{array}\right\}$.
    \item $R_n(13/2) = \{w \in R_n : \text{w is weakly increasing}\}$.
    \item $R_n(12/3) = \{w \in R_n : \text{w has initial run $1\dots m$ and $a_{m+1} = \dots = a_n \leq m$ }\}$.
    \item $R_n(123) = \{w \in R_n : \text{w has no element repeated more than twice}\}$.
\end{enumerate}
\eth

With this theorem in hand, we now explain how to read the block, spread, and dimension statistics of a partition from its RGF. To do so it will help to introduce the following notation. Given an RGF $w$ and a letter $l$ of $w$, let $\first(l)$ and $\last(l)$ be the indices of the first and last occurrence of $l$ in $w$, respectively. For example, if $w=11231$, then $\first(1) =1$, $\last(1) = 5$, and $\first(3) = \last(3) = 4$.
\blem
Let $\pi \in \Pi_n$ with $w=w(\pi)$. Then

\begin{enumerate}
    \item $\block(\pi) = \max(w)$,
    \item $\spread(\pi) = \sum\limits_{i = 1}^{\max(w)} (\last(i)- \first(i))$,
    \item $\dim(\pi) = \sum\limits_{i = 1}^{\max(w)} (\last(i) - \first(i)+1)$.
\end{enumerate}
\elem

\begin{proof}
The first equality follows directly from the bijection between $R_n$ and $\Pi_n$.
The spread statistic is given by the blockwise sum of the difference between the maximum and minimum elements in each block. To track the maximum and minimum elements in a block $B_{i}$, we must find the first occurrence of $i$ in the RGF and the last occurrence of $i$ in the $RGF$, and take their difference. Hence, 
$$
\spread(w) = \sum_{i = 1}^{\max(w)} (\text{last$(i)$ $-$ first$(i)$}).
$$
The last statement follows from $\dim(w) = \spread(w) + \block(w)$.
\end{proof}

For example, below we display the partitions of length three together with their corresponding RGFs and block, spread, and dimension statistics.
\begin{center}
    \begin{tabular}{|c|c|c|c|c|}
    \hline
    $\pi$ & $w(\pi)$ & $\block(\pi)$ & $\spread(\pi)$ & $\dim(\pi)$ \\
    \hline
    $1/2/3$ & $123 $ & $3$ & $0$ & $3$ \\
    \hline
    $1/23$ & $122$ & $2$ & $1$ & $3$ \\
    \hline
    $12/3$ & $112$ & $2$ & $1$ & $3$ \\
    \hline
    $13/2$ & $121$ & $2$ & $2$ & $4$ \\
    \hline
    $123$ & $111$ & $1$ & $2$ & $3$ \\
    \hline
    \end{tabular}
\end{center}

In the remainder of this section we will combine Theorem 2.1 and Lemma 2.2 to study $\SB_n(P;q,t)$ and $\Dim(P;q)$ for patterns of length three. To simplify notation somewhat we will use $[n]_q$ to denote the $q$-analogue of $n$, 
$$
[n]_q = 1+q+q^2+ \dots +q^{n-1}.
$$
\subsection{Patterns of Length 3}
We start by dealing with the pattern $12/3$. We obtain a $(q,t)$-analog of the well known identity $1+2+\dots+(n-1) = \binom{n}{2}$.
\bth
We have
    $$\SB_{n}(12/3;q,t) = t^n+\sum_{i=1}^{n-1}t^{i}q^{n-i}[i]_q$$ and $$\Dim(12/3;q) = q^n\left(1+\sum_{i=1}^{n-1}[i]_q\right).$$
\eth

\begin{proof}
    From Theorem 2.1, every RGF in $R_n(12/3)$ is of the form $12...mk...k$ where $1 \leq k \leq m$. From this, we can then characterize every RGF first by its maximum element, then by the value of the constant string at the end. 
    
    If $w$ is strictly increasing, $w=12\dots n$, then the trailing constant string is empty, and hence $\block(w) = n,$ $\spread(w) = 0$. Alternatively, suppose an RGF has maximum value $m$, and a string of length $n-m$ and value $k$ at the end. Then it is clear that $\block(\pi) = m$ and $\spread(w) = n - k$. We have
\begin{equation*}
\sum_{w\in R_n(12/3)}q^{\spread(w)}t^{\block(w)}=t^n+\sum_{m=1}^{n-1}\sum_{k=1}^m t^{m}q^{n-k}
\end{equation*}
and simplifying gives the desired result.
\end{proof}
Next we will examine the pattern $1/23$. Partitions avoiding this pattern are in bijection with those avoiding the partition $12/3$; we will exhibit this bijection explicitly, show that it respects the statistics of interest, and then apply Theorem 2.3 to prove the following.
\bth
We have
    $$\SB_{n}(1/23;q,t) = t^n+\sum_{i=1}^{n-1}t^{i}q^{n-i}[i]_q$$ and $$\Dim(1/23;q) = q^n\left(1+\sum_{i=1}^{n-1}[i]_q\right).$$
\eth

\begin{proof}
As discussed previously, we will exhibit a statistic preserving bijection $\tau:R_n(1/23) \to R_n(12/3)$. To do this, let $w\in R_n(1/23)$. If $w=12\dots n$, we set $\tau(w) = w$. Otherwise, by Theorem 2.1, $w$ is obtained by inserting a single $1$ into a word of the form $1^l23\dots m$. Suppose the additional $1$ is inserted at index $l+k$, for $1\leq k\leq n-l$. We map $w$ to the word $12\dots m(m-k+1)\dots(m-k+1)$.

Since $m\geq m-k+1$ we know $\tau(w) \in R_n(12/3)$, and hence our map is well defined. It is also straightforward to invert $\tau$, so it suffices to show $\tau$ preserves spread and block. Since $\tau$ does not alter the maxima of $w$ we know $\block(w) = \block(\tau(w))$. Finally, for $w\in R_n(1/23)$, the only letter which contributes to spread is the letter $1$. Alternatively, if $w\in R_n(12/3)$ then the only letter which contributes to $\spread(w)$ is the letter appearing in the trailing constant string. One readily verifies that $\tau$ respects these contributions, verifying our theorem. 
\end{proof}

Next we move to the pattern $13/2$. Any RGF in $R_n(13/2)$ is weakly increasing, which makes this computation suitable for recursion. Along the way we obtain a $(q,t)$-analog of the identiy $1+1+2+4+\dots + 2^{n-1} = 2^n$.
\bth
We have $\SB_n(13/2;q,t) = \Dim_n(13/2;q) =1$ and, for $n\geq1$,
    $$\SB_n(13/2;q,t) = (q+t)^{n-1}t$$ 
    and 
    $$\Dim_n(13/2;q) = 2^{n-1}q^{n}$$
\eth

\begin{proof}
Let $w\in R_n(13/2)$ for $n\geq1$. By Theorem 2.1, $w$ is weakly increasing. Suppose $w$ ends in a constant string of length $l$, for $1\leq l \leq n$. Removing this string from $w$ yields an RGF $v$ of length $n-l$ which is also weakly increasing, and hence contained in $R_{n-l}(13/2)$. This process yields a bijection 
$$
R_n(13/2) \to \coprod_{k=0}^{n-1}R_k(13/2). 
$$
Moreover, if $v\in R_k(13/2),$ then mapping $v$ to the word $w$ obtained by appending $n-k$ copies of $\max(v)+1$ to $v$ has the following effect:
\begin{align*}
    \block(w)&= 1+\block(v)\\
    \spread(w)&=n-k-1 + \spread(v). 
\end{align*}
We obtain 
$$
\SB_n(13/2;q,t) = \sum_{k=0}^{n-1}q^{n-k-1}t\SB_{k}(13/2;q,t),
$$
which simplifies to the formulas in the theorem statement via standard generatingfunctionology.
\end{proof}

The next pattern we examine is $1/2/3$. Despite the fact that the partitions avoiding $1/2/3$ are simply those whose associated RGFs only use the letters $1$ and $2$, it is quite difficult to get a clean formula for $SB_n(1/2/3;q,t)$. We include these results mainly for completeness.  

\begin{theorem}
For $n\geq 2$ we have
\begin{align*}
    \SB_n(1/2/3;q,t)&=q^{n-1}t + t^2\left((n-2)q^{n-1}+(n-1)q^{n-2}+\sum_{i=2}^{n-2}\sum_{j=i+1}^{n-1}2^{j-i-1}q^{j-i}(1+q^{n-1})\right),\\
    \Dim_n(1/2/3;q)&= (n-2)q^{n+1}+nq^n+\sum_{i=2}^{n-2}\sum_{j=i+1}^{n-1}2^{j-i-1}q^{j-i}(1+q^{n-1}).
\end{align*}
\end{theorem}
\begin{proof}
To prove this theorem we partition $R_n(1/2/3)$ into several sets, and compute the distribution of spread and block over these individual sets. First we examine words in $R_n(1/2/3)$ which end with the letter $1$. We iterate over the position of the first and last $2$ in $w$ (dealing with the case $w=11\dots 1$ separately). If the first $2$ in $w$ has position $i$ and the last $2$ in $w$ has position $j$, then 
$$
\spread(w) = n-1+ j-i.
$$
If $i=j$ then such a $w$ is necessarily unique; if $i<j$ then there are $2^{j-i-1}$ words of this form, coming from the choice of letters between index $i$ and $j$. Thus
$$
\sum_{\substack{w\in R_n(1/2/3)\\w_n = 1}}q^{\spread(w)}t^{\block(w)}= q^{n-1}t+q^{n-1}t^2\left((n-2)+\sum_{i=2}^{n-2}\sum_{j=i+1}^{n-1}2^{j-i-1}q^{j-i}\right).
$$
Now suppose $w$ ends in the letter $2$. If $w$ is of the form $1^l2^{n-l}$, then $\spread(w) = n-2$. If $w$ is not of this form, then the first $2$ in $w$ precedes the last $1$. Iterating over the possible indices of these values gives $$
\sum_{\substack{w\in R_n(1/2/3)\\w_n = 2}}q^{\spread(w)}t^{\block(w)}= t^2\left((n-1)q^{n-2}+\sum_{i=2}^{n-2}\sum_{j=i+1}^{n-1}2^{j-i-1}q^{n+j-1-i}\right).
$$
Summing over all possible cases gives the desired result.
\end{proof}
Numerical calculations seem to show that the coefficients of $\Dim_n(1/2/3;q)$ are given by OEIS sequence A120933, and more generally that the coefficients of $\SB_n(1/2/3;q,t)$ are closely related to OEIS A296612; perhaps one could find a cleaner statement for our result by making these relations explicit.

The last result of this section deal with the pattern $123$. RGFs of partitions avoiding $123$ will contain no element more than twice. This restriction is much less rigid than the ones seen previously, which makes working with these partitions much more difficult. Accordingly, we only provide partial information on the individual statistics over this class. Namely, our theorem describes the words in $R_n(123)$ which maximize the spread statistic.

\bth
Let $w$ be an RGF in $R_n(123)$ which maximizes the spread statistic. Then
$$
\spread(w) = \left\lfloor\tfrac{n}{2}\right\rfloor\left\lceil\tfrac{n}{2}\right\rceil
$$
and 
$$
\block(w) = \lceil\tfrac{n}{2}\rceil.
$$
Moreover, $w$ is of the form $12 \dots \lceil\frac{n}{2}\rceil \sigma$, where $\sigma$ is a permutation of the set $\{1, 2, \dots, \lfloor\frac{n}{2}\rfloor\}$.
\eth

For example, let $n=11$. The RGF $w  = 12345653142$ is of the above form, its spread is $30 = 5 \cdot 6$, and its block is 6.
\prove
We prove this theorem in three steps. First, we show that if $u = 12 \dots \lceil\frac{n}{2}\rceil \sigma$ and $v = 12 \dots \lceil\frac{n}{2}\rceil \tau$ where $\sigma$ and $\tau$ are two different permutations of $\{1, 2, \dots, \lfloor\frac{n}{2}\rfloor\}$, then $\spread(u) = \spread(v)$. Then, we show that this construction indeed gives the maximum spread in $R_n(123)$. Finally, we calculate the spread of these spread-maximizing RGFs. The fact that $\block(w) = \lceil\frac{n}{2}\rceil$ follows directly from the strictly increasing sequence $123 \dots \lceil\frac{n}{2}\rceil$ at the start of $w$. 

Let $u$ and $v$ be as above. We show that $\spread(u) - \spread(v) = 0$. By definition, 
\begin{align*}
    \spread(u) - \spread(v) = \sum_{i = 1}^{\max(u)} (\text{last$(i)$ $-$ first$(i)$}) - \sum_{j = 1}^{\max(v)} (\text{last$(j)$ $-$ first$(j)$})
\end{align*}
We know that $\max(u) = \max(v) = \lceil\tfrac{n}{2}\rceil$, and that the first $\lceil\frac{n}{2}\rceil$ terms are the same in both $u$ and $v$. Accordingly, 
\begin{align*}
\spread(u)-\spread(v) = \sum_{i=1}^{\lceil\frac{n}{2}\rceil} \last(i)-  \sum_{j=1}^{\lceil\frac{n}{2}\rceil} \last(j).
\end{align*}
The righthand summations are merely reorderings of the same sum, and hence this difference vanishes. 

Now we show that this indeed gives the maximum spread. Let $\pi\in \Pi_n(123)$ be a partition which maximizes spread. Clearly $\pi$ cannot contain two singleton blocks; if it did, we could replace the singletons with their union to increase spread. Translating into RGFs, a maximizing RGF $w$ can contain at most one unique letter.

 Next, suppose $w\in R_n(123)$ maximizes spread and that the first $\lceil\frac{n}{2}\rceil$ elements are not strictly increasing. By the definition of RGFs there must exist indices $i,j$ such that $i < j \leq \lceil\frac{n}{2}\rceil$ but $w_i = w_j$. Note that in this scenario this common value is less than $\lfloor \frac{n}{2}\rfloor$. Importantly, in order for $w$ to be in $R_n(123)$, the letter $\lceil\frac{n}{2}\rceil$ \emph{must} appear in $w$ (a word of length $n$ which can repeat values at most twice must use at least $n/2$ letters). Combining this with the fact that $w$ is an RGF, we know the value $k=\max(w_1,\dots,w_j)+1$ is in $w$ as well. Transposing the first occurence of $k$ in $w$ with $w_j$ will increase spread, contradicting our assumption. 
 
From the previous two paragraphs, we know that if $w$ maximizes spread then it must contain at most one unique letter, and $w$ must start by strictly increasing up to the letter $\lceil\frac{n}{2}\rceil$. The only such words which also avoid the pattern $123$ are those described in our theorem.
 
 Now we finally calculate the spread of such an RGF. Let $w = 12 \dots \lceil\frac{n}{2}\rceil 12 \dots \lfloor\frac{n}{2}\rfloor$. We calculate the spread of this RGF, and it will be the same for all $w$ of the form $123 \dots \lceil\frac{n}{2}\rceil \sigma$. The spread is easy to calculate, since there are $\lfloor\frac{n}{2}\rfloor$ elements that contribute to the spread, and they each contribute $\lceil\frac{n}{2}\rceil$. So the spread is $\lceil\frac{n}{2}\rceil \lfloor\frac{n}{2}\rfloor$.
\done
As a fun corollary of the previous result we obtain the following well known identites:
\begin{align*}
    &1+3+\dots + 2k-1 = k^2,\\
    &2+4+\dots + 2k = k(k+1).
\end{align*}
This can be seen by applying Theorem 2.7 to the following partitions:
\begin{align*}
1(2k)/2(2k-1)/\dots/k(k+1)\;\;\text{ and }\;\;1(2k+1)/2(2k)/\dots/k(k+2)/(k+1).
\end{align*}
\begin{remark}
The partitions appearing in the proof of Theorem 2.7 also maximize spread over $\Pi_n$. If $n$ is odd, $n=2k+1$, there is one other class of partitions which do so. Using the notation from the above proof, they are the partitions whose associated RGF has the form 
$$
w = 12\dots kk\pi_1\dots\pi_k
$$
where $\pi$ is a permutation of $[k]$ (note that the letter $k$ will appear 3 times in such a word, hence it is not encompassed by Theorem 2.7).
\end{remark}
\section{Noncrossing Partitions}
Now we consider the set of partitions that avoid the pattern $13/24$. These partitions are called \emph{noncrossing} and have a rich combinatorial and algebraic history. The size of the set $\Pi_n(13/24)$ is given by the $n$th Catalan number, 
$$
C_n = \frac{1}{2n+1}\binom{2n}{n}.
$$
See the work of Armstrong for more information \cite{Arm09} on these partitions.  We will continue our analysis of the spread, block, and dimension statistics over this avoidance class, and in doing so obtain $(q,t)$-analogs of the standard Catalan recursion. Our main result is Theorem 3.3; to prove this we need characterizations of $R_n(13/24),$ which were provided by Campbell et al. in \cite{REURGF}. We restate their results for completeness.

\blem[\cite{REURGF}, Lemmas 5.1 and 5.2]
For a partition $\pi$, the following are equivalent:
\begin{itemize}
    \item[(1)] $\pi$ avoids $13/24$,
    \item[(1)] the RGF $w(\pi)=w_1\dots w_n$ avoids 1212 (as a subword pattern),
    \item[(2)] there are no $xyxy$ subwords in w,
    \item[(3)] if $w_{i} = w_{i'}$ for some $i < i'$ then, for all $j' > i'$, either $w_{j'} \leq w_{i'}$ or $w_{j'} > \max\{w_{1}, \dots, w_{i'}\}$.
\end{itemize}
\elem

For the next corollary and the following theorem we need the following notation. For a word $w=w_1\dots w_n$, let $1w$ denote the word obtained by prepending a $1$ to $w$. For an integer $k$, let $(w+k)$ denote the word whose $i$th letter is $w_i+k$. For example, if $w = 12134$, then $1(w+1) = 123245$. Similarly, if $u$ and $v$ are words then $uv$ is $u$ prepended to $v$.

\begin{corollary}[\cite{REURGF}, Lemma 5.3]
If $w$ is in $R_{n}(13/24)$ then both $1w$ and $1(w+1)$ are in $R_{n+1}(13/24)$.
\end{corollary}

The next theorem gives a recursive formula for $SB_n(13/24;q,t)$. The structure of this proof will closely follow that of Theorem 5.4 in \cite{REURGF}.
\begin{theorem}
We have
\begin{align*}
\SB_{0}(13/24;q,t) &= 1 \\
\SB_{1}(13/24;q,t) &= t
\end{align*}
and for $n\geq 2$
\begin{equation*}
\SB_{n}(13/24;q,t) = tSB_{n-1}(13/24;q,t) + \sum_{k=2}^{n}q^{k-1}\SB_{k-2}(13/24;q,t)\SB_{n-k+1}(13/24;q,t).
\end{equation*}
Similarly, $\Dim_0(13/24;q)=1$, $\Dim_1(13/24;q) = q$, and for $n\geq 2$
$$
\Dim_n(13/24;q) = q\Dim_{n-1}(13/24;q) +\sum_{k=2}^nq^{k-1}\Dim_{k-2}(13/24;q)\Dim_{n-k+1}(13/24;q).
$$
\end{theorem}

\begin{proof}
The initial conditions are readily verified; we will focus on proving the recursion. To prove this theorem we divide the set $R_n(13/24)$ into disjoint subsets, then split our analysis of the block and spread statistics individually over these sets. The sets of interest are as follows:
$$
X = \{w \in R_n(13/24) : w_{1} = 1 \text{ and there are no other 1s in $w$}\} $$
and, for $k\geq 2$,
$$Y_k = \{w \in R_n(13/24) : w_1 = w_k = 1\text{ and }w_j>1\text{ for }1<j<k\}.$$
That these sets partition $R_n(13/24)$ is clear.

We will begin by looking at $X$. We can describe $X$ as 
$$
X = \{w\in R_n(13/24): w=1(u+1) \text{ for some }u\in R_{n-1}(13/24)\},
$$
and from this description we obtain a bijection $X\to R_{n-1}(13/24)$. That this is a bijection follows from the definition of RGFs, together with Lemma 3.1 and Corollary 3.2.

If $w = 1(u+1)$, then $\block(w) = \block(u)+1$ since $\max(w) = \max(u) + 1$. Further, $\spread(w) = \spread(u)$, as incrementing each letter of $u$ by the same amount does not change its spread, and nor does prepending a (unique) 1 onto it. Thus
$$
\sum_{w\in X}q^{\spread(w)}t^{\block(w)} = t\SB_{n-1}(13/24,q,t).
$$
    
Next, we examine the sets $Y_k$. For $w\in Y_k$, we claim that $w$ is of the form $w = 1(u+1)1v$ where $u \in R_{k-2}(13/24)$, $\st(1v) \in R_{n-k+1}(13/24)$, and if $v_i \neq 1$, then $v_i > \max(u) + 1$. 

The first two requirements are clear. For the third, if there exists $v_i$ such that $1 < v_i \leq \max(u) + 1$ then there must exist $u_j$ such that $u_j$ + 1 = $v_i$. However, then we will have an $xyxy$ subword with $x = 1$ and $y = u_j + 1 = v_i$, which would then imply that $w \notin R_n(13/24)$ by Lemma 3.1.

In particular, this establishes a bijection 
$$
Y_k\to R_{k-2}(13/24)\times R_{n-k+1}(13/24);
$$
the inverse map sends $(u,v)\to 1(u+1)v'$, where 
\begin{equation}
v'_i = \begin{cases}1\text{ if }v_i=1\\v_i+\max(u)\text{ otherwise.}\end{cases}
\end{equation}
Now we examine the behaviour of block and spread under this map. Take $w = 1(u+1)1v\in Y_k$. If $1v$ consists solely of $1$s, then $\max(w) = 1+\max(u)$. Alternatively, the maximum value of $w$ will be found in $1v$; in any case, 
$$
\block(w) = \block(u) + \block(\st(1v)).
$$

Now consider the spread statistic. Since $(u+1)$ and $1v$ are disjoint, the spread of $w$ is \emph{almost} the sum $\spread(u)+\spread(\st(1v));$ however, this forgets the impact of prepending the leading $1$. As $u$ has length $k-2$, the effect of this is to increase spread by $k-1$. The correct formula is therefore 
$$
\spread(w) = \spread(u)+\spread(\st(1v))+k-1.
$$
Hence summing over $Y_k$ results in 
\begin{align*}
\sum_{w \in Y_k}q^{\spread(w)}t^{\block(w)} &= q^{k-1} \left(\sum_{u \in R_{k-2}(13/24)}q^{\spread(u)}t^{\block(u)}\right)\left(\sum_{\st(1v) \in R_{n-k+1}(13/24)}q^{\spread(v)}t^{\block(1v)}\right)\\
&=q^{k-1}\SB_{k-2}(13/24;q,t)\SB_{n-k+1}(13/24;q,t).
\end{align*}
Summing over $k$, and remembering the contribution of $X$, gives our desired result.
\end{proof}
The recursions in Theorem 3.3 have many siblings in the realm of Catalan combinatorics; we will not discuss any of them currently, as Section 5 is entirely devoted to this. 
\section{Multiple Patterns}
In this section we examine generating functions for statistics taken over set partitions avoiding two patterns. These calculations will provide refinements of work of Goyt \cite{Goy08} and lead to interesting connections to Fibonacci numbers and Motzkin paths. 

To start we will examine set partitions which avoid two patterns of size 3. Note that a simplified Erdos-Szekeres argument shows that there are finitely many set partitions avoiding both $1/2/\dots/n$ and $12\dots m$; in particular, we will not examine the finitely many partitions avoiding both $1/2/3$ and $123$. 

To help with our analysis, we display in Table 1 the following results from \cite{Goy08}, which discuss characterizations of avoidance classes in this context.
\begin{center}
\bgroup
\def\arraystretch{1.9}
\begin{tabular}{|c|c|}

\hline
    Avoidance Class & Corresponding Partitions \\
    \hline
    $\Pi_n(1/2/3,1/23)$ & $12\dots n$, $12\dots(n-1)/n$, $12\dots (n-2)n/(n-1)$\\
    \hline
        $\Pi_n(1/2/3,13/2)$ & $12\dots n$, $12\dots k/(k+1)\dots n$, $1\leq k < n$ \\
    \hline
        $\Pi_n(1/2/3,12/3)$ & $12\dots n$, $1/2\dots n$, $13\dots n/2$\\
    \hline
        $\Pi_n(1/23,13/2)$ & $12\dots n$, $12\dots k/(k+1)/\dots/n$, $1\leq k<n$\\
    \hline
        $\Pi_n(1/23,12/3)$ & $12\dots n$, $1n/2/\dots/(n-1)$, $1/2/\dots /n$\\
    \hline
        $\Pi_n(1/23,123)$ & $1/2/\dots/n$, $1k/2/\dots/(k-1)/(k+1)/\dots/n$, $1<k\leq n$\\
    \hline
        $\Pi_n(13/2,12/3)$ & $1/2/\dots /k\dots n$, $1\leq k\leq n$\\
    \hline
        $\Pi_n(13/2,123)$ & $\pi$ has blocks of size $\leq 2$ and $w(\pi)$ is weakly increasing\\
    \hline
        $\Pi_n(12/3,123)$ & $1/2/\dots/n$, $1/2/\dots/kn/\dots /(n-1)$, $1\leq k <n$\\
    \hline

\end{tabular}
\egroup
\\
\vspace{5pt}
Table 1: Partitions which avoid two patterns of length 3, from \cite{Goy08}
\end{center}
As is evident, many generating functions are straightforward to compute given the information in Table 1. The following calculations use ideas already introduced, and hence we merely state them.
\begin{proposition}
We have the following generating functions for $n\geq 1$:
\begin{align*}
    &\SB_n(1/2/3,1/23;q,t) = q^{n-1}(t +t^2)+ q^{n-2}t^2,\\
    &\SB_n(1/2/3,13/2;q,t) = q^{n-1}t+ (n-1)q^{n-2}t^2,\\
    &\SB_n(1/2/3,12/3;q,t) = q^{n-1}(t+t^2) + q^{n-2}t^2,\\
    &\SB_n(1/23,13/2;q,t) = q^{n-1}t+ \sum_{k=1}^{n-1}q^{k-1}t^{n-k+1},\\
    &\SB_n(1/23,123;q,t) = q[n-1]_qt^{n-1}+t^n,\\
    &\SB_n(13/2,12/3;q,t) = q^n\sum_{k=1}^n(q^{-1}t)^k,\\
    &\SB_n(12/3,123;q,t) =q[n-1]_qt^{n-1}+t^n.
\end{align*}
The generating functions for the dimension index are easily obtained via specialization. \qed
\end{proposition}
Note that the generating function $\SB_n(13/2,123)$ is missing from the previous proposition. It turns out that this is the most interesting duo of length $3$ patterns to avoid. The cardinality of $\Pi_n(13/2,123)$ is given by the $n$th Fibonacci number, see Proposition 2.10 of \cite{Goy08}. Moreover, one obtains interesting $q$-Fibonacci numbers by studying the distribution of statistics over this set; see the work of Goyt and Sagan, who examine the ``fundamental statistics'' of Wachs and White over this set \cite{GS09}. The analog in this situation is as follows:
\begin{proposition}
We have $\SB_0(13/2,123;q,t)=1,$ $\SB_1(13/2,123;q,t) = t$, and for $n\geq 2$
$$
\SB_n(13/2,123;q,t) = t\SB_{n-1}(13/2,123;q,t)+qt\SB_{n-2}(13/2,123;q,t).
$$
\end{proposition}
\begin{proof}
We recurse on the last block $B$ of $\pi\in \Pi_n(13/2,123)$. As $w(\pi)$ is weakly increasing, $B$ is an interval. As $\pi$ avoids $123$, this interval is size $1$ or size $2$. Thus removing $B$ yields a partition in either $\Pi_{n-1}(13/2,123)$ or $\Pi_{n-2}(13/2,123)$. Studying how the statistics behave under this recursion gives the desired result.
\end{proof}
This result has several connections to known facts; for instance, there is a clear map from partitions in $\Pi_n(13/2,123)$ to tilings of a $1\times n$ rectangle using $1\times 1$ squares (monominoes) and $1\times 2$ rectangles (dominoes). Specializing $t=1$ in $\SB_n(13/2,123;q,t)$ lets us examine the number of dominoes in these tilings, and recovers a well known Fibonacci identity:
$$
\SB_{n}(13/2,123,q,1) = \sum_{k=\lceil n/2\rceil}^{n}\binom{k}{n-k}q^k.
$$
Alternatively we may set $q=1$, which counts tilings by the total number of tiles used. We recover OEIS Sequence A129710 in doing so; surely more connections can be made by examining other specializations. The interested reader should examine Goyt and Sagan's $q$-Fibonacci numbers \cite{GS09}, those of Carlitz \cite{Car74} \cite{Car75}, Cigler \cite{Cig03}, or many others.

We end this section by combining a pattern of length $3$ with the pattern $13/24$. We start with $\Pi_n(123,13/24)$; in doing so we obtain $q$-analogs of the Motzkin numbers. 

By combining Theorem 2.1 and Lemma 3.1, it is easy to characterize $\Pi_n(123,13/24)$ as simply the noncrossing partitions in which each block has size $\leq 2$. Following the proof of Theorem 3.3, we have the following $(q,t)$-Motzkin recursion:
\begin{theorem}
We have $\SB_0(123,13/24;q,t) = 1$, $\SB_1(123,13/24;q,t) = t$, and for $n\geq 2$, 
$$
\SB_n(123,13/24;q,t) = t\SB_{n-1}(123,13/24;q,t) + \sum_{k=0}^{n-2}q^{k+1}\SB_k(123,13/24;q,t)\SB_{n-k-2}(123,13/24;q,t).
$$
\end{theorem}
\begin{proof}
We proceed as in Theorem 3.3, partitioning $\Pi_n(123,13/24)$ into the sets
\begin{align*}
    X &= \{w\in R_n(111,1212):w_1=1\text{ and there are no other }1s\text{ in }w\}\\
    Y_k &= \{w\in R_n(111,1212):w_1=w_k=1\text{ and }w_j>1\text{ for }1<j<k\}.
\end{align*}
Analyzing these sets using the ideas discussed in Theorem 3.3 shows that $X$ is in bijection with $R_{n-1}(111,1212)$. Similarly, $Y_k$ is in bijection with $R_{k}(111,1212)\times R_{n-k-2}(111,1212)$. Be careful to note the change in the second index compared to Theorem 3.3; the bijection in this context is $1(u+1)1v\to (u,\st(v))$, in other words the latter $1$ \emph{ is not }included in the word $v$. This is because, as $w$ avoids $111$, the second $1$ is the \emph{last} occurence of $1$ in $w$. Analyzing the behavior of block and spread under these maps gives the desired recursion.
\end{proof}
Unsurprisingly, the generating function of Theorem 4.3 has many interesting connections to the literature. For instance, it arises by specializing the Motzkin Tunnel Polynomials of Barnabei, Bonetti, and Castronuovo at $z=x_i = t$, $y_i = q$; see Theorem 5 of \cite{Tunnel}. Alternatively, we may specialize to $q=1$ to obtain the Motzkin polynomials (OEIS A055151) which arise, among other places, in Marberg's study of the so called ``poor'' noncrossing partitions \cite{Mar}. Specializing to $t=1$ gives an alternate form of Motzkin polynomials which count Motzkin paths by area (OEIS A129181). It would be interesting to give bijective proofs of these results, in analogy to Theorem 5.10 of \cite{REURGF}.

We end this section by considering $\Pi_n(1/2/3,13/24)$. It is straightforward that an RGF associated to such a partition has the form $w = 1^k2^j1^{n-k-j}$. From this, the following is not difficult.
\begin{proposition}
We have 
$$
\SB_n(1/2/3,13/24) = tq^{n-1}+t^2\left((n-1)q^{n-2}+q^{n-1}\sum_{k=2}^{n-1}[n-k]_q\right).
$$
\qed
\end{proposition}
\section{Connections to $321$-Avoiding Permutations}
In this section we show how the methods of Section 3 can be used to obtain results on 321-avoiding permutations. To do so, we briefly recall a few notions (for a more extended introduction to permutation patterns, we refer the reader to a reference such as \cite{Kit11}). 

A \emph{permutation} of $[n]$ is a bijection $\pi:[n]\to[n]$. We will write $\pi$ in one line form as $\pi = \pi_1\dots \pi_n$, where $\pi_i = \pi(i)$. Let $\mathfrak{S}_n$ denote the permutations of length $n$; for example, 
$$
\mathfrak{S}_3 = \{123,132,213,231,312,321\}.
$$
As with set partitions there is a natural notion of containment in permutations. Namely, given a permutation $\pi$ of length $n$ and a permutation $\tau$ of length $m$, we say $\pi$ contains $\tau$ as a pattern if there is a subword $\pi_{i_1}\dots \pi_{i_m}$ in $\pi$ with $i_1<\dots<i_m$ such that 
$$
\pi_{i_j}<\pi_{i_k} \text{ if and only if }\tau_j<\tau_k.
$$

So for example, the permutation $\pi = 1432$ contains the permutation $321$; this is exhibited by examining the last three letters of $1432$. However, the permutation $\sigma = 1342$ avoids the pattern $321$ as it contains no decreasing subsequence of length $3$. We let $\Av_n(\pi)$ denote the length $n$ permutations which avoid the pattern $\pi$.

The study of permutation patterns has an extensive history, dating back to MacMahon \cite{Mac15} and to Knuth \cite{Knu05}. As with set partitions, there is much interest in computing the size of avoidance classes of permutations. And the use of combinatorial statistics allows one to refine these counts; this was pioneered in work of Dokos, Dwyer, Johnson, Sagan and Selsor \cite{DDJSS12}. 

This section will concern itself with five statistics on permutations; they are of combinatorial interest but also arise in the algebraic theory of symmetric groups. Let $\pi=\pi_1\dots \pi_n$ be a permutation. An index $i$ is a \emph{left-to-right maxima} if $\pi_j<\pi_i$ for all $j<i$ (this condition is vacuously satisfied by the first index). An index $i$ is a \emph{fixed point} if $i=\pi_i$. A pair of indices $(i,j)$ form an \emph{inversion} if $i<j$ and $\pi_i>\pi_j$. Finally, an index $i$ is a \emph{descent} if $\pi_i>\pi_{i+1}$. We may now define the following statistics:
\begin{align*}
    \LRM(\pi) &= \#\{i:i\text{ is a left-to-right maxima }\}\\
    \inv(\pi) &= \#\{(i,j):(i,j)\text{ is an inversion }\}\\
    \fix(\pi) &= \#\{i:i\text{ is a fixed point }\}\\
    \des(\pi) &= \#\{i:i\text{ is a descent }\}\\
    \maj(\pi) &= \sum_{i\text{ a descent}} i.
\end{align*}

Sparked by Conjecture 3.2 and Question 3.4 of \cite{DDJSS12}, Cheng, Elizalde, Kasraoui, and Sagan have provided (among other things) Catalan recursions for the inversion polynomials for $321$-avoiding permutations, as well as for the joint generating functions for descents and the major index \cite{CEKS13}. Their proofs involve beautiful connections of $321$-avoiding permutations with lattice paths and polyominoes. 

The purpose of this section is to give a quick alternative proof of certain results of \cite{CEKS13} using noncrossing partitions. Namely, we will present Catalan recursions for the generating functions 
\begin{align*}
    I_n(q,t,x)&:= \sum_{\pi \in Av_n(321)}q^{\inv(\pi)}t^{\LRM(\pi)}x^{\fix(\pi)}\\
    M_n(q,t,x)&:= \sum_{\pi \in Av_n(321)}q^{\maj(\pi)}t^{\des(\pi)}x^{\LRM(\pi)}.
\end{align*}To do so, we take advantage of the following characterization of $321$-avoiding permutations. Given $\pi\in \mathfrak{S}_n$, define binary vectors 
\begin{align*}
\pos(\pi) &=(p_1,\dots,p_n)\\
\val(\pi) &=(v_1,\dots,v_n) 
\end{align*}
with 
$$
p_i = \begin{cases}1\text{ if }i\text{ is a left-right maximum in }\pi\\0\text{ otherwise} \end{cases}
$$
and 
$$
v_i = \begin{cases}1\text{ if there is a left-to-right maxima }j\text{ with } \pi_j=i\\0\text{ otherwise.} \end{cases}
$$
In other words, $\val(\pi)$ describes the values of left-right maxima in $\pi$, and $\pos(\pi)$ determines the positions of these maxima.

The following is a folklore lemma, documented in \cite{CEKS13}:
\begin{lemma}{[CEKS13, Lemma 2.1]}
The assignment $\pi \to (pos(\pi),val(\pi))$ induces a bijection between $Av_n(321)$ and the set of pairs of binary vectors $(p_1\dots p_n,v_1\dots v_n)$ satisfying
\begin{itemize}
\item The number of $1$s in $p_1\dots p_n$ equals the number of $1$s in $v_1\dots v_n$, and
\item For any index $1\leq i\leq n-1$, the number of $1$s in $p_1\dots p_i$ is \emph{strictly greater} than the number of $1$s in $v_1\dots v_{i-1}$.
\end{itemize} \qed
\end{lemma}
Note that for the second condition to apply for $i=1$, we require that $p_1=1$. We will call such a pair of binary sequences a \emph{ballot pair}. The bijection for $n=3$ is reproduced below:
\begin{align*}
123&\to(111,111)\\
132&\to(110,101)\\
213&\to(101,011)\\
231&\to(110,011)\\
312&\to(100,001).
\end{align*}

The first step of this section is to establish a bijection between noncrossing partitions and ballot pairs; via Lemma 5.1, this will establish a bijection between noncrossing partitions and $321$-avoiding permutations. 

Let $w=w_1\dots w_n$ be the RGF of a noncrossing partition. Call a letter $w_i$ a \emph{first} if $w_j\neq w_i$ for $j<i$ and a \emph{last} if $w_j\neq w_i$ for $j>i$. With this terminology, we can define a map $T$ from noncrossing partitions to ballot pairs by sending
$$
w_1\dots w_n\to (f_1\dots f_n, l_1\dots l_n),
$$ 
with 
$$
f_i = \begin{cases}1\text{ if }w_i\text{ is a first,}\\0\text{ otherwise}\end{cases}
$$
and 
$$
l_i = \begin{cases}1\text{ if }w_i\text{ is a last,}\\0\text{ otherwise.}\end{cases}
$$
As an example, below we show the action of $T$ on $R_3(13/24)$: 
\begin{align*}
111&\to(100,001)\\
112&\to(101,011)\\
122&\to(110,101)\\
121&\to(110,011)\\
123&\to(111,111).
\end{align*}
Using Lemma 3.1, it is not hard to show the following:
\begin{lemma}
The map $T$ induces a bijection between noncrossing partitions of $[n]$ and ballot pairs of length $n$. 
\end{lemma}
\begin{proof}
We sketch how to reconstruct a noncrossing partition from its corresponding ballot pair. Let $(p_1\dots p_n,v_1\dots v_n)$ be a ballot pair of length $n$. We construct $w=w_1\dots w_n\in R_n(13/24)$ iteratively as follows. Start with the empty RGF, and set $L_0 = \emptyset$. Here the sets $L_j$ denote the set of ``available'' letters at any given step, i.e. the letters whose first occurrences have been placed but whose last occurrences have not yet been established. Having constructed $w_1\dots w_{j}$ and the set of available letters $L_j$, we determine $w_{j+1}$ and $L_{j+1}$ as follows. 
\begin{itemize}
\item If $p_{j+1} = v_{j+1} = 1$, set $w_{j+1} = \max\{w_1,\dots,w_j\}+1$ and $L_{j+1} = L_j$. In this case, the ``first'' letter is also a ``last,'' so our available letters do not change.

\item If $p_{j+1} = 1$ and $v_{j+1} = 0$, set $w_{j+1} = \max\{w_1,\dots,w_j\}+1$ and \\ $L_{j+1}=L_j\cup\{\max\{w_1,\dots,w_j\}+1\}$.

\item If $p_{j+1} = 0$ and $v_{j+1} = 1$, set $w_{j+1} = \max L_j$ and $L_{j+1} = L_{j+1}\setminus \max L_j$.

\item If $p_{j+1} = v_{j+1} =0$, set $w_{j+1} = \max L_j$ and $L_{j+1} = L_j$.
\end{itemize}
This process is well defined by the definitions of a ballot pair. It is easy to see that this process yields an RGF with no $xyxy$ patterns, and by Lemma 3.1 this implies $w$ is noncrossing. That this map is an inverse to $T$ follows from inspection.
\end{proof}
By combining Lemma 5.1 and 5.2, we have a bijection between noncrossing partitions of $[n]$ and $321$-avoiding permutations of length $n$. Abusing notation, we also call this map $T$.  To keep track of fixed points, descents, and the major index on permutations, we introduce the following statistics on partitions. These statistics were reverse engineered to provide insight into their permutation counterparts, so they might look quite ad hoc at first glance.  

Given a partition $\pi$, let $w=w_1\dots w_n$ denote its associated RGF. Maintaining notation as for permutations, call an index $i$ a left-to-right maxima if $w_j<w_i$ for all $j<i$ (this condition is vacuously satisfied by the first index). We call an index $i$ a \emph{checkpoint} if it is a left-to-right maxima {and} if $w_i<w_j$ for all $j>i$ (the second half of this condition is vacuously satisfied by the last index).  Finally, an index $i$ is an \emph{apex} if it is a left to right maxima in $w$ and if $w_i \geq w_{i+1}$ (the second half of this condition \emph{is not} vacuous; the last index in a word is never an apex).

For example, in the word $12213454$ the indicies $1,2,5,6,7$ are left-to-right maxima, the index $5$ is a checkpoint, and the indices $2$ and $7$ are apices. 
\begin{lemma}
Let $\pi$ be a noncrossing partition with corresponding RGF $w=w_1\dots w_n$ and with $T(w) = a_1\dots a_n$. Then 
\begin{itemize}
\item $\spread(w) = \inv(T(w))$,
\item $\block(w) = \LRM(T(w))$,
\item $i$ is a checkpoint in $w$ if and only if $i$ is a fixed point in $T(w)$,
\item $i$ is an apex in $w$ if and only if $i$ is a descent in $T(w)$.
\end{itemize}
\end{lemma}
\begin{proof}
Throughout this proof, let $(p_1\dots p_n,v_1\dots v_n)$ denote the corresponding ballot pair of $w$ and $T(w)$.

For the first assertion, note that a LRM in an RGF is equivalently an index of a first occurrence of some letter. Thus 
$$
\spread(w) = \sum_{i\text{ a LRM of }w}(\last(w_i)-i).
$$
Similarly, as $T(w)=a_1\dots a_n$ is a $321$-avoiding permutation we have 
$$
\inv(T(w)) = \sum_{i\text{ a LRM of }T(w)}(a_i - i).
$$
This last equality is not necessarily trivial unless one has experience with permutation patterns; it follows from the fact that in a $321$-avoiding permutation, the letters at indices which are not $\LRM$ must form a strictly increasing sequence. 

Since the bijection $T$ exchanges the set of last occurrences in $w$ with the set of LRM values in $T(w)$, both sums are equal to 
$$ \left(\sum_{v_i = 1}i\right) - \left(\sum_{p_j = 1}j\right).$$

For the second assertion, we simply observe that both $\block(w)$ and $\LRM(T(w))$ count the number of $1$s in $p_1\dots p_n$. 

For the third assertion, we claim that $w_i$ is a checkpoint if and only if $p_i = v_i=1$ and the prefix pair $(p_1\dots p_{i-1}, v_1\dots v_{i-1})$ is a ballot pair. Indeed the first condition is implied by $w_i$ being unique; the condition that $w_i$ is a checkpoint is equivalent to having $\last(l)<i$ for every letter $l<w_i$, which implies the second. Translating this to permutations, the fact that $p_i = v_i = 1$ implies that the index $i$ is a LRM in $T(w)$, and that the letter $i$ is a value of a LRM in $T(w)$. The prefix condition assures us that these conditions imply $a_i = i$.

Finally, in a $321$-avoiding permutation the letters at indices which are not left-to-right maxima must be strictly increasing. In particular, $a_i>a_{i+1}$ if and only if $i$ is a left-to-right maxima, but $i+1$ is not. Thus the descents in $T(w)$ are precisely the indices $i$ such that $p_i = 1$ and $p_{i+1}=0$. Thus in $w$, $w_i$ is the first occurence of a letter, and $w_{i+1}$ is not. This implies $w_{i+1}\leq w_i$ by the growth restrictions of RGFs, i.e. $i$ must be an apex. 
\end{proof}
We now establish the main results of this section. Theorem 5.4 below should be compared to Theorem 8.4 of \cite{CEKS13}, which is proved using continued fractions. 
\begin{theorem}[\cite{CEKS13} Theorem 8.4]
The polynomials $I_n(q,t,x)$ satisfy $I_0(q,t,x) = 1$ and, for $n\geq 1$, 
$$
I_n(q,t,x) = tx I_{n-1}(q,t,x) + \sum_{j=2}^nq^{j-1}I_{j-2}(q,t,1)\big(I_{n-j+1}(q,t,x) - t(x-1)I_{n-j}(q,t,x)\big).
$$
\end{theorem}
\begin{proof}
By the previous three lemmas, it suffices to work with the distribution of spread, block and checkpoints on noncrossing partitions of length $n$, defining 
\begin{align*}
    \cp(w) &= \#\{i:i\text{ is a checkpoint in }w\}.
\end{align*}
Our goal is to compute the generating function 
$$
\sum_{w\in R_n(13/24)}q^{\spread(w)}t^{\block(w)}x^{\cp(w)}.
$$
To do so, we use the recursive argument developed in Theorem 3.3. Recall the definition of the sets 
$$
X = \{w_1\dots w_n \in R_n(13/24): w_i>1\text{ for }i>1\}
$$
and, for $i=2,3,\dots,n$, 
$$
Y_k = \{w_1\dots w_n\in R_n(13/24): w_k = 1\text{ and }w_j>1 \text{ for }1<j<k\}.
$$
As before, $X$ is in bijection with $R_{n-1}(13/24)$, with the map given by 
$$
u = u_1\dots u_{n-1} \to 1(u+1).
$$
Examining the behavior of the three statistics of interest under this map yields
$$
\sum_{w\in X}q^{\spread(w)}t^{\block(w)}x^{\cp(w)} = t\cdot x\cdot I_{n-1}(q,t,x).
$$

Next, for $w\in Y_k$ let us write $w = 1(u+1)1v$, with $u\in R_{k-2}(13/24)$ and $\st(1v)\in R_{n-k+1}(13/24)$. As in Theorem 3.3, $\spread(w) = \spread(u)+\spread(\st(1v))+k-1$ and $\block(w) = \block(u)+\block(\st(1v))$. The checkpoint statistic is slightly more subtle;  the relation of $\cp(w)$ to $\cp(u)$ and $\cp(\st(1v))$ depends on whether or not there is a $1$ in the word $v$. 

We will get around the previous issues by an application of the Inclusion-Exclusion Principle. Let $V_k$ be the set
\begin{align*}
    V_k &= \{w\in R_{n-k+1}: w\text{ contains a single 1 }\}.
\end{align*}
Writing $w\in Y_k$ as $w=1(u+1)1v$ induces a bijection between $Y_k$ and the disjoint union
$$
(R_{k-2}\times (R_{n-k+1}\setminus V_k)) \coprod (R_{k-2}\times V_k).
$$
If $w = 1(u+1)1v$, then $\cp(w) = \cp(\st(1v))$ if $\st(1v)\in R_{n-k+1}\setminus V_k$, and $\cp(w) = \cp(\st(1v))-1$ if $\st(1v)\in V_k$. This is because if $\st(1v)$ has a unique $1$ (necessarily at the first index), then the index $1$ will be a checkpoint in $\st(1v)$ but will no longer be a checkpoint in $1(u+1)1v$.

A variant of Inclusion Exclusion gives that the generating function for the three statistics over $Y_k$ factors as a product of the polynomial
$$
q^{k-1}\left(\sum_{u\in R_{k-2}(13/24)}q^{\spread(u)}t^{\block(u)}\right)
$$
with 
$$
\left(\sum_{\st(1v)\in R_{n-k+1}(13/24)}q^{\spread(\st(1v))}t^{\block(\st(1v))}x^{\cp(\st(1v))} - (1-x^{-1}) \sum_{\st(1v)\in V_k}q^{\spread(\st(1v))}t^{\block(\st(1v))}x^{\cp(\st(1v))} \right).
$$
In other words, we have taken a naive count over all $u\in R_{k-2}$ and $\st(1v)\in R_{n-k+1}$, and then modified it with the appropriate correction where it is needed (i.e. with respect to $V_k$).

But by an argument that is now standard, we can put $V_k$ in bijection with $R_{n-k}$ by sending $w\in R_{n-k}$ to $1(w+1)$. Examining how our statistics are affected by this map, we obtain
$$
\sum_{w\in Y_k}q^{\spread(w)}t^{\block(w)}x^{\cp(w)} = q^{k-1}I_{k-2}(q,t,1)\big(I_{n-k+1}(q,t,x) - t(x-1)I_{n-k}(q,t,x)\big).
$$
Summing over $k$ completes the proof.
\end{proof}

Of course, upon specialization of variables we recover, for instance, Theorem 1.1 of \cite{CEKS13} as well, which provides a recursion for inversions and left-to-right maxima. 

We end by showing how to mildly generalize part of Theorem 6.2 of \cite{CEKS13}, which examines descents and the major index. Cheng, Elizalde, Kasraoui, and Sagan prove their theorem using the theory of polyominoes. Our purpose in reexamining this computation is merely to continue exploring the connection between noncrossing partitions and $321$-avoiding permutations.

\begin{theorem}
We have $M_0(q,t,x) = 1$ and, for $n\geq1$,
$$
M_n(q,t,x) =xM_{n-1}(q,qt,x)+\sum_{k=2}^n \Big(M_{k-1}(q,t,x)+x(q^{k-1}t-1)M_{k-2}(q,t,x)\Big)M_{n-k}(q,q^{k}t,x).
$$
\end{theorem}
\begin{proof}
Define 
$$
\ap(w) = \#\{i:i\text{ is an apex in }w\}
$$
and 
$$
\maj(w) = \sum_{i \text{ an apex }} i.
$$
It suffices to determine the distribution of apices, major index, and block over $R_n(13/24)$, as by Lemma 5.3 we have
$$
\sum_{w\in R_n(13/24)}q^{\maj(w)}t^{\ap(w)}x^{\block(w)}= M_n(q,t,x).
$$ This proof is similar in spirit to Theorems 3.3 and 5.4, but requires a different recursive argument. The ideas should be familiar by now, so we merely sketch the ideas. Partition $R_n(13/24)$ into the sets 
$$
R_n(13/24) = \coprod_{k=1}^nY_k,
$$
with 
$$
Y_k:= \{w\in R_n(13/24) : \text{ the last occurrence of the letter }1\text{ in }w\text{ has index }k\}.
$$
Similarly to the proof of Theorem 3.3, $Y_1$ is in bijection with $R_{n-1}(13/24),$ and 
$$
\sum_{w\in Y_1}q^{\maj(w)}t^{\ap(w)}x^{\block(w)}= xM_{n-1}(q,qt,x).
$$
The sets $Y_k$, for $k\geq 2$, are in bijection with the Cartesian products
$$
Y_k\leftrightarrow R_{k-1}(13/24)\times R_{n-k}(13/24);
$$
the map exhibiting this sends $(u,v)\in R_{k-1}(13/24)\times R_{n-k}(13/24)$ to $w= u1(v+\max(u)+1)$. Any index which is an apex in $u$ or $v$ will promote to an apex of $w$; additionally, if $u$ ends in a unique letter, then this will provide an additional apex of $w$ which was not an apex of $u$. Accordingly,
$$
\ap(w) = \begin{cases}\ap(u)+\ap(v)+1\text{ if }u\text{ ends in a unique letter,}\\
\ap(u)+\ap(v) \text{ otherwise}.
\end{cases}
$$
Keeping track of the position of these apices gives 
$$
\maj(w) = \begin{cases}\maj(u)+\maj(v)+k\ap(v)+k-1\text{ if }u\text{ ends in a unique letter,}\\
\maj(u)+\maj(v)+k\ap(v) \text{ otherwise}.\end{cases}
$$
Applying the same Inclusion-Exclusion argument as in Theorem 5.4 and summing over $k$ gives the desired result. 
\end{proof}
Specializing the variable $x=1$ yields the first recursive formula presented in Theorem 6.2 of \cite{CEKS13}. It would be interesting if a clean recursion could be found which combines Theorems 5.4 and 5.5, i.e. to compute the joint distribution of our five statistics of interest over the $321$-avoiding permutations. However, it is not clear how to do so neatly; if one recurses on the index of the second letter $1$ in $R_n(13/24)$, it becomes hard to keep track of apices. Alternatively, it is difficult to keep track of the spread statistic when one recurses on the last occurrence of the letter $1$. In principle one could do so, but it is unclear if the resulting formula can be simplified into anything worth looking at. 
\section{Future Directions}
As is evident, many interesting connections in combinatorics can be found by studying a combination of combinatorial statistics and combinatorial patterns. We end with several ideas one could examine in this area. 

\textbf{Longer Patterns:} The most obvious extension would could make to this article is to continue studying the distribution of these statistics over avoidance classes of longer patterns. For example, Sagan provides closed formulae for the number of partitions avoiding $12/3/\dots /m$ and for the number of partitions avoiding $1/23\dots m$ in \cite{Sag10}. Can one generalize our arguments to those settings? 

\textbf{Other Classes of Partitions:} There are several other natural classes of set partitions, which are not defined via the notion pattern avoidance defined above. What can one say about the distribution of dimension, spread, and block over these classes? For example, one could work with the notion of pattern avoidance in terms of restricted growth functions as is done in \cite{REURGF}. In this context, the generating function for the dimension index taken over RGFs avoiding the pattern $112$ is a sum of Gaussian binomial coefficients. It would be interesting to explore this in detail.

Alternatively, one could work with other combinatorially defined sets of partitions, such as the nonnesting partitions. For an introduction to such objects and their relation to noncrossing partitions, see, for instance, \cite{AST13}.

\textbf{Machine Learning:} Can machine learning be used to examine combinatorial patterns, in any context? Such computations have found use in computational algebraic geometry and theoretical physics \cite{MLearn}; analogs in the combinatorial setting could be useful in further developing combinatorial databases, such as Tenner's Database of Permutation Pattern Avoidance \cite{DPPA}.

\textbf{Poset Structure:}
Sagan proposed a version of these questions in \cite{Sag10}; we repeat the topic here, in the hopes of someone looking into it! Let $\Pi = \bigcup_n \Pi_n$. In analog to the permutation pattern poset, we can put a partial order on $\Pi$ by saying $\sigma\leq \pi$ if $\pi$ contains $\sigma$ as a pattern. What can we say about this poset? The topology of the permutation pattern poset has received much interest \cite{top}; what can be said of the topology of the set partition pattern poset? The permutation pattern poset contains infinite antichains \cite{SB00,ABV13}; does the set partition pattern poset?

\textbf{Connections to Permutations:}
With Section 5 in mind, can one find more connections between pattern avoidance in set partitions and pattern avoidance in permutations (or, more generally, pattern avoidance in other contexts)? Following a comment of Kyle Petersen on OEIS entry A055151, a potential start would be to connect Theorem 4.3 to descents and peaks in 231-avoiding permutations. Alternatively, one could try to relate partitions in $\Pi_n(123,13/24)$ to permutations avoiding $321$ and the so called \emph{barred pattern} $3\overline{1}24$; see \cite{Barred} for more information.

\bibliography{spref}{}
\bibliographystyle{alpha}
\nocite{}
\end{document}